\documentclass[11pt,reqno]{amsart}

\usepackage{amssymb}
\usepackage{amsfonts}
\usepackage[margin=1.5in]{geometry}
\usepackage{cite}
\usepackage{tikz}

\usetikzlibrary{automata}

\newcommand{\rank}{\mathop{\mathrm{rank}}}
\newcommand{\krank}{\mathop{\boldsymbol{r}}}
\newcommand{\bigast}{\mathop{\hbox{\huge $\ast$}}}
\newcommand{\Span}{\mathop{\mathrm{Span}}\nolimits}
\newcommand{\core}{\mathop{\mathrm{Core}}\nolimits}
\newcommand{\red}{\mathop{\mathrm{red}}\nolimits}

\newcommand{\inv}{^{-1}}
\newcommand{\p}{\varphi}

\newcommand{\pinv}{\p \inv}
\def\e<{\leq _{E}}

\def\ov#1{\ensuremath{\overline {#1}}}

\def\til#1{\ensuremath{\widetilde {#1}}}

\def\malce{\mathop{\hbox{$\bigcirc$\kern-8.5pt\raise1pt
\hbox{\scriptsize$m$}\kern1.5pt}}}
\newcommand{\dom}{\mathop{\mathrm {dom}}\nolimits}
\newcommand{\ran}{\mathop{\mathrm {ran}}\nolimits}

\def\1sk{^{(1)}}

\def\to{\rightarrow}

\def\data{\ifcase\month\or January\or February \or March\or April\or May
\or June\or July\or August\or September\or October\or November \or
December\fi\space\number\day, \number\year}

\def\Thmname{Theorem}
\def\Propname{Proposition}
\def\Lemmaname{Lemma}
\def\Definitionname{Definition}
\newtheorem{Thm}{\Thmname}%[section]
\newtheorem{Prop}[Thm]{\Propname}
\newtheorem{Lemma}[Thm]{\Lemmaname}

\newtheorem{Cor}[Thm]{Corollary}

\title{On a conjecture of Karrass and Solitar}
\author{Benjamin Steinberg}
\address{Department of Mathematics\\ City College of New York\\ Convent Ave at
138th St.\\ New York, New York 10031}
\thanks{This work was partially supported by a grant from the Simons Foundation (\#245268 to Benjamin Steinberg).} \email{bsteinberg@ccny.cuny.edu}

\date{March 29, 2013; revised November 6, 2013}



\begin{document}
\begin{abstract}
We settle an old conjecture of Karrass and Solitar by proving that a finitely generated subgroup of
a non-trivial free product $G = A\ast B$ has finite index if and only if it
intersects non-trivially each non-trivial normal subgroup of $G$. This holds, more generally, for subgroups of finite Kurosh rank.
\end{abstract}

\maketitle

\section{Introduction}
Karrass and Solitar deduced in~\cite{KS1}, as an easy consequence of
M.~Hall's theorem~\cite{Hall2}, that a finitely generated subgroup of
a free group $F$ has finite index if and only if it intersects
non-trivially every non-trivial normal subgroup of $F$. This result was then forgotten for a number of years until it was revisited by Arzhantseva, who in her 1998 Ph.~D.\ thesis proved a quantitative version of this result: if $H$ is a finitely generated subgroup of infinite
index in a free group $F$, then the normal closure of almost any finite set of elements of $F$ misses $H$. Her approach was via Stallings graphs and small cancellation theory. The result was published in~\cite{Goulnara}.
In the proceedings of a 1998 conference in Lincoln, Nebraska, another proof of the result of Karrass and Solitar was given by  Ivanov and Schupp~\cite{IS}, who apparently were unaware of the result of Karrass and Solitar. The result of Karrass and Solitar does appear in the 1977 book of Lyndon and Schupp from twenty years earlier~\cite[Chapter~1, Proposition~3.17]{LyndonandSchupp}.  The proof of Ivanov and Schupp also used Stallings graphs and small cancellation theory.  Shortly after, Kahrobaei~\cite{Delaram} offered another simple proof using Marshall Hall's theorem.

In a paper~\cite{KS2}, published in the same year as~\cite{KS1}, Karrass and Solitar conjectured that a finitely generated subgroup of a free product $G$ of non-trivial
groups has finite index if and only
if it intersects non-trivially each non-trivial normal subgroup of $G$.  We settle here this conjecture affirmatively.  Our proof follows that of Arzhantseva and of Ivanov and Schupp, but uses covering spaces instead of Stallings graphs and makes use of an analogue of the Spelling Theorem, due to Duncan and Howie~\cite{DuncanHowie}, instead of classical small cancellation theory. In fact, we prove more generally the following theorem.

\begin{Thm}\label{main}
Let $G=A\ast B$ be a free product of non-trivial groups $A$ and $B$ and let $H\leq G$ be a subgroup of finite Kurosh rank.  Then $[G:H]<\infty$ if and only if $H\cap N\neq \{1\}$ for every normal subgroup $\{1\}\neq N\lhd G$.
\end{Thm}

We shall recall the notion of Kurosh rank in Section~\ref{core}.  Suffice it to say for our purposes that finitely generated subgroups have finite Kurosh rank.  Consequently, the conjecture of Karass and Solitar is now a theorem.

\begin{Thm}\label{main2}
Let $G=A\ast B$ be a free product of non-trivial groups $A$ and $B$.  Then a finitely generated subgroup of $G$ has finite index if and only if it intersects non-trivially each non-trivial normal subgroup of $G$.
\end{Thm}

Theorem~\ref{main2} can be seen as a strong generalization of the theorem of B.~Baumslag~\cite{BBaumslag} that any finitely generated subgroup of a free product of non-trivial groups, which contains a non-trivial normal subgroup is of finite index (since in a free product of non-trivial groups, any two non-trivial normal subgroups intersect non-trivially).

\section{The core $2$-complex of a subgroup}\label{core}
The purpose of this section is to associate to each subgroup $H\leq G=A\ast B$ a $2$-complex $\core(H)$ that plays the role of the Stallings graph of a subgroup in the theory of free groups.  Several core-like constructions have been considered before in the literature~\cite{BurnsCompany,Ivanov,sykiotis}, but ours seems to be a new variation, in particular because it is a $2$-complex.

Fix groups $A\neq \{1\}\neq B$ and let $G=A\ast B$.  If $C$ is a group, put $\til C=C\setminus \{1\}$. The free monoid on a set $X$ is denoted $X^*$.  Sometimes, we will write $(x)$ for $x\in \til A\cup \til B$ if we want to emphasize that we are viewing $x$ as an element of $(\til A\cup \til B)^*$ rather than as an element of $G$.  A word $\gamma$ in $(\til A\cup \til B)^*$ is called \emph{reduced} if it contains no factors of the form $(x)(y)$ with both $x,y\in \til A$ or $x,y\in \til B$.  It is well known~\cite[Chapter~IV, Theorem~1.2]{LyndonandSchupp} that any element of $G$ is represented by a unique reduced word and that any word can be put into its reduced form using the following elementary rewriting rules:
\begin{align}\label{rule1}
 (x)(x\inv)\rightarrow 1, &\quad \text{for $x\in \til A\cup \til B$}\\ \label{rule2} (a)(a')\rightarrow (aa'), & \quad \text{for $a,a'\in \til A$, $a'\neq a\inv$}\\ \quad \label{rule3} (b)(b')\rightarrow (bb'), &\quad \text{for $b,b'\in \til B$, $b'\neq b\inv$.}
\end{align}

For $\gamma\in (\til A\cup \til B)^*$, we write $\red(\gamma)$ for the unique reduced word equivalent to it in $G$.  Abusing notation, denote by $\red(g)$ the unique reduced word representing $g\in G$.  Let $\ell(g)$ be the length of the word $\red(g)$. A word $\gamma\in (\til A\cup \til B)^*$ is said to be \emph{cyclically reduced} if $\gamma^2$ is reduced, or equivalently, if $\gamma^n$ is reduced for all $n\geq 0$. Let us say that $g\in G$ is \emph{cyclically reduced} if $\red(g)$ is cyclically reduced.  There is an involution on $(\til A\cup \til B)^*$ defined in the usual way: $[(x_1)\cdots (x_n)]\inv = (x_n\inv)\cdots (x_1\inv)$. The involution preserves the properties of being reduced and of being cyclically reduced.

If $C$ is a group, let $K(C)$ be the $2$-complex with a single vertex $v_0$, edge set $\til C$ (consisting of loops) and with $2$-cells of the form
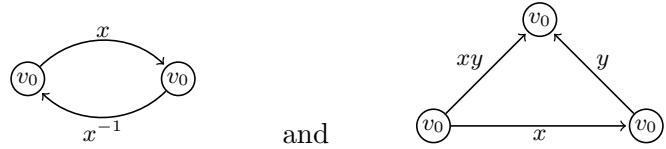
\begin{center}\begin{tikzpicture}[->,shorten >=1pt,%
auto,node distance=2cm,semithick,
inner sep=1pt,bend angle=45]
\tikzset{every state/.style={minimum size=0pt}}.
\node[state] (A) {\footnotesize{$v_0$}};
\node[state] (B) [right of=A] {\footnotesize{$v_0$}};
\tikzstyle{every node}=[font=\footnotesize]
\path (A) edge [bend left] node [above]  {$x$} (B)
      (B) edge [bend left] node [below] {$x\inv$} (A);
\end{tikzpicture}\qquad\quad and\qquad\quad
\begin{tikzpicture}[->,shorten >=1pt,%
auto,node distance=2cm,semithick,
inner sep=1pt,bend angle=45]
\tikzset{every state/.style={minimum size=0pt}}.
\node[state] (A) {\footnotesize{$v_0$}};
\node[state] (C) [above right of=A] {\footnotesize{$v_0$}};
\node[state] (B) [below right of=C] {\footnotesize{$v_0$}};
\tikzstyle{every node}=[font=\footnotesize]
\path (A) edge node  [below] {$x$}  (B)
      (B) edge node  [above right] {$y$}  (C)
      (A) edge node  {$xy$} (C);
\end{tikzpicture}
\end{center}
where $x\in C$ and $y\neq x\inv$.
Of course, $\pi_1(K(C),v_0)\cong C$.  Let $K=K(A)\vee K(B)$ be the wedge product.  Then $\pi_1(K(A)\vee K(B),v_0)\cong G$.

If $X$ is a $2$-complex and $V\subseteq X^{(0)}$ is a set of vertices, then the subcomplex $X[V]$ \emph{induced} by $V$ has $1$-skeleton $X[V]^{(1)}$ consisting of $V$, together with all edges of $X$ between vertices of $V$.  A $2$-cell $c$ of $X$ belongs to $X[V]$ if the attaching map $\partial c\to X^{(1)}$ has image contained in $X[V]^{(1)}$. A subcomplex of the form $X[V]$ for some set of vertices $V$ is called an \emph{induced subcomplex}.

Fix a subgroup $H\leq G$ for the remainder of this section and let \[\p\colon (\til K_H,w_0)\to (K,v_0)\] be the pointed covering map with $\p_*(\pi_1(\til K_H,w_0))=H$. By a path in a $2$-complex, we shall always mean an edge path. Each path in $\til K_H$ is labelled by a word in $(\til A\cup \til B)^*$.  Let us say that the path is \emph{reduced} if its label is reduced. By usual covering space theory, a path $p$ is homotopy equivalent, rel base points, to a unique reduced path $\red(p)$, whose label, moreover, is the reduced form of the label of $p$.

\begin{Lemma}\label{stayininduced}
Let $X\subseteq \til K_H$ be an induced subcomplex and let $p$ be a path in $X$.  Then $\red(p)$ is also a path in $X$ and $p$ is homotopy equivalent, rel base points, to $\red(p)$ inside of $X$.
\end{Lemma}
\begin{proof}
It suffices to show that the result of applying one of the rewriting rules \eqref{rule1}--\eqref{rule3} to the label of $p$ is the label of a path contained in $X$, which is homotopic to $p$ in $X$, rel base points.  If $p$ has a factorization $p=p_1efp_2$ where $e,f$ are edges with respective labels $x,x\inv$ with $x\in \til A\cup \til B$, then the path $ef$ bounds a $2$-cell $c$ in $\til K_H$ by construction of $K$ and because $\p$ is a covering map.  See Figure~\ref{figurerule1}.  Since $X$ is an induced subcomplex, $c$ belongs to $X$.  Thus $p$ is homotopy equivalent to $p_1p_2$ in $X$, rel base points.  This takes care of rewriting rule~\eqref{rule1}.

\begin{figure}[bt]
\begin{center}
\begin{tikzpicture}[->,shorten >=1pt,%
auto,node distance=2cm,semithick,
inner sep=1pt,bend angle=45]
\tikzset{every state/.style={minimum size=0pt}}.
\node[state] (A) {};
\node[state] (B) [right of=A] {};
\node[state] (C) [right of=B] {};
\node[state] (D) [above of=B]   {};
\tikzstyle{every node}=[font=\footnotesize]
\path (A) edge node [above]  {$p_1$} (B)
      (B) edge  node [above] {$p_2$} (C)
      (B) edge [bend left] node [left]  {$x$} (D)
      (D) edge [bend left] node [right] {$x\inv$} (B);
\end{tikzpicture}
\end{center}
\caption{Rewriting rule of type \eqref{rule1}\label{figurerule1}}
\end{figure}
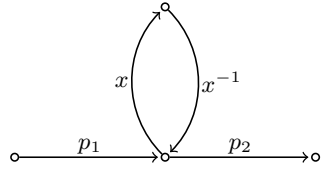

We now handle rewriting rule~\eqref{rule2}, as the case of \eqref{rule3} is analogous. Suppose $p=p_1efp_2$ where $e,f$ are edges with respective labels $a,a'\in \til A$ with $a'\neq a\inv$.  Let $v$ be the initial vertex of $e$ and $w$ the terminal vertex of $f$.  Because $\p$ is a covering and $(a)(a')(aa')\inv$ bounds a $2$-cell in $K$, there is an edge $e'$ from $v$ to $w$ labelled by $aa'$. See Figure~\ref{figurerule2}. As $X$ is an induced subcomplex and $v,w$ are vertices of $X$, the edge $e'$ belongs to $X$ and so $p_1e'p_2$ is a path in $X$.  Moreover, since $\p$ is a covering $ef(e')\inv$ bounds a $2$-cell $c$, which furthermore belongs to $X$ because $X$ is an induced subcomplex.  It follows that $p$ is homotopic to $p_1e'p_2$ in $X$, rel base points.
\begin{figure}[bt]
\begin{center}
\begin{tikzpicture}[->,shorten >=1pt,%
auto,node distance=2cm,semithick,
inner sep=1pt,bend angle=45]
\tikzset{every state/.style={minimum size=0pt}}.
\node[state] (A) {};
\node[state] (B) [right of=A] {};
\node[state] (C) [below right of=B] {};
\node[state] (D) [above right of=C] {};
\node[state] (E) [right of=D] {};
\tikzstyle{every node}=[font=\footnotesize]
\path (A) edge node [above]  {$p_1$} (B)
      (B) edge  node [below left] {$a$}   (C)
      (C) edge  node [below right]  {$a'$}  (D)
      (B) edge  node [above] {$aa'$} (D)
      (D) edge  node [above] {$p_2$} (E);
\end{tikzpicture}
\end{center}
\caption{Rewriting rule of type \eqref{rule2}\label{figurerule2}}
\end{figure}
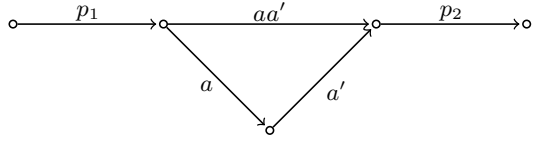
This completes the proof.
\end{proof}

As a corollary, we obtain that the restriction of the covering map $\p$ to any induced subcomplex is $\pi_1$-injective.

\begin{Cor}\label{pi1inj}
If $X\subseteq \til K_H$ is an induced subcomplex containing $w_0$, then $\p_*\colon \pi_1(X,w_0)\to H$ is injective.
\end{Cor}
\begin{proof}
If $q$ is a loop in $X$ at $w_0$, then $q$ is homotopic in $X$, rel base points, to $\red(q)$ by Lemma~\ref{stayininduced}.  Thus without loss of generality we may assume that $q$ is reduced.  But then $\p_*([q])=1$ if and only if $q$ is an empty path by the normal form theorem for free products~\cite[Chapter~IV, Theorem~1.2]{LyndonandSchupp}.
\end{proof}

Let $\emptyset\neq S\subseteq H$.  We say that a vertex $v$ of $\til K_H$ is in the \emph{span} of $S$ if there is an element $s\in S$ such that the lift of $\red(s)$ at $w_0$ passes through $v$. The subcomplex of $\til K_H$ induced by the vertices in the span of $S$ is denoted $\Span(S)$ and is called the subcomplex \emph{spanned} by $S$. Note that $w_0\in \Span(S)$ and $\Span(S)$ is connected. Also observe that if $S$ is finite, then $\Span(S)$ has finitely many vertices.
Let us define the core of $\til K_H$ to be the subcomplex $\core(H)=\Span(H)$ spanned by $H$ itself.

The next proposition sets up a Galois correspondence between non-empty subsets of $H$ and induced subcomplexes of $\til K_H$ containing the base point $w_0$.

\begin{Prop}\label{galoisconnection}
Let $\emptyset\neq S\subseteq H$ and let $X$ be an induced subcomplex of $\til K_H$ containing $w_0$. Then $\Span(S)\subseteq X$ if and only if $S\subseteq \p_*(\pi_1(X,w_0))$.
\end{Prop}
\begin{proof}
By construction $S\subseteq \p_*(\pi_1(\Span(S),w_0))$ and so the only if statement is clear. For the converse, if $g\in S$, then there is a loop $p$ in $X$ at $w_0$ with $\p_*([p])=g$.  But then Lemma~\ref{stayininduced} implies that $\red(p)$, which is the lift of $\red(g)$ to $w_0$, is contained in $X$.  We conclude that $\Span(S)\subseteq X$.
\end{proof}

Taking $S=H$ we obtain the following corollary.

\begin{Cor}\label{smallestinduced}
The $\core(H)$ is the smallest induced subcomplex $X$ of $\til K_H$ containing $w_0$ such that $\p_*(\pi_1(X,w_0))=H$.
\end{Cor}

The next corollary shows that $\core(H)$ is determined by a generating set of $H$.

\begin{Cor}\label{generates}
Let $\emptyset\neq S\subseteq H$.  Then $\Span(S)=\Span(\langle S\rangle)$.  In particular, if $\emptyset\neq S$ generates $H$, then $\core(H)=\Span(S)$.
\end{Cor}
\begin{proof}
Trivially, $\Span(S)\subseteq \Span(\langle S\rangle)$.  For the converse, observe that $S\subseteq \p_*(\pi_1(\Span(S),w_0))$ and hence $\langle S\rangle \subseteq \p_*(\pi_1(\Span(S),w_0))$.  Proposition~\ref{galoisconnection} now yields that $\Span(\langle S\rangle)\subseteq \Span(S)$.
\end{proof}

\begin{Cor}\label{finitecorefg}
Suppose that $H$ is finitely generated.  Then $\core(H)$ has finitely many vertices.
\end{Cor}

Finiteness of the vertex set of the core is all that we use in the proof of Theorem~\ref{main}, so we now prove that $\core(H)$ is finite if and only if $H$ has finite Kurosh rank.  We recall here the definition of Kurosh rank; see~\cite{BurnsCompany,Ivanov,sykiotis} for details.

By the Kurosh theorem, we know that $H\leq G=A\ast B$ has a free product decomposition of the form
\begin{equation}\label{kuroshdecomp}
H=\left[\bigast_{i\in I}s_iH_is_i\inv\right]\ast\left[\bigast_{j\in J}t_jK_jt_j\inv\right]\ast F(H)
\end{equation}
where $H_i\leq A$, for $i\in I$, $K_j\leq B$ for $j\in J$, $F(H)$ is a free group intersecting no conjugate of $A$ or $B$, and the $s_i,t_j\in G$ run over certain sets of double coset representatives of $H\backslash G/A$ and $H\backslash G/B$, respectively.  The \emph{Kurosh rank} $\krank(H)$ of $H$, which may be infinite, is defined (relative to the splitting $G=A\ast B$) by
\[\krank(H) = \rank(F(H))+\#\{i\in I\mid H_i\neq \{1\}\}+\#\{j\in J\mid K_j\neq \{1\}\}.\]  One can show that this does not depend on any of the choices made in the Kurosh decomposition~\cite[Lemma~4]{Ivanov}.

There is an alternative description of the Kurosh rank in terms of Bass-Serre theory.  Let $T$ be the Bass-Serre tree for the free splitting $G=A\ast B$.  Then, because the edge stabilizers are trivial, there is a unique minimal $H$-invariant subtree $T_H$ for any subgroup $H\leq G$.  One has \[\krank(H)=\rank(\pi_1(H\backslash T_H))+\#\{Hv\in H\backslash T_H\mid G_v\neq \{1\}\}\] where $G_v$ is the stabilizer of the vertex $v\in T_H$.  The Kurosh rank is finite precisely when $H\backslash T_H$ is finite.  See~\cite{sykiotis} for details.

We shall require just the ``only if'' statement of the following theorem; therefore, the other implication will only be sketched.

\begin{Thm}\label{finitecore}
Let $H\leq G=A\ast B$.  Then $H$ has finite Kurosh rank if and only if $\core(H)$ has finitely many vertices.
\end{Thm}
\begin{proof}
Assume first that $\krank(H)<\infty$ and let \eqref{kuroshdecomp} be its Kurosh decomposition. Let $X$ be a finite generating set for $F(H)$ and let $I' = \{i\in I\mid  H_i\neq \{1\}\}$, $J'=\{j\in J\mid  K_j\neq \{1\}\}$; note that $I',J'$ are finite.  Then Corollary~\ref{generates} yields \[\core(H)=\Span\left(X\cup \bigcup_{i\in I'}s_iH_is_i\inv\cup \bigcup_{j\in J'}t_jK_jt_j\inv\right).\]  Clearly $X$ contributes only finitely many vertices to $\core(H)$.  It suffices to show that each $s_iH_is_i\inv$ with $i\in I'$ and $t_jK_jt_j\inv$ with $j\in J'$ contributes finitely many vertices to $\core(H)$.  We handle only the first case, as the second is identical.

Let $q_i$ be the lift of $\red(s_i)$ at $w_0$ and let $C_i$ be the subcomplex of $\til K_H$ induced by the vertices visited by the path $q_i$; note that $C_i$ has finitely many vertices.
Let $1\neq a\in H_i\leq A$.  Let $p$ be the lift at $w_0$ of the path $\red(s_i)a\red(s_i)\inv$ in $K$.  Then $p$ is a loop at $w_0$ (as $s_ias_i\inv \in H$) and so by uniqueness of path lifting, $p=q_ieq_i\inv$ where  $e$ is the edge out of the endpoint of $q_i$ labeled by $a$.  In particular, $e$ must be a loop edge.  Thus $p$ is contained in $C_i$ and, therefore, $\red(p)$ is contained in $C_i$ by Lemma~\ref{stayininduced}.  We conclude that the lift at $w_0$ of the reduced form of each element of $s_iH_is_i\inv$ is contained in $C_i$ and hence $s_iH_is_i\inv$ contributes only finitely many vertices to $\core(H)$.

That the finiteness of $\core(H)^{(0)}$ implies $\krank(H)<\infty$ is immediate from the next lemma.
\end{proof}

\begin{Lemma}\label{findKurosh}
Let $X\subseteq \til K_H$ be a connected induced subcomplex containing $w_0$ with $\p_*(\pi_1(X,w_0))=H$.  Let $\Gamma(X)$ be the bipartite graph with
\begin{align*}
V(\Gamma(X)) &= \pi_0(\pinv(K(A))\cap X)\cup \pi_0(\pinv (K(B))\cap X)\\
E(\Gamma(X)) &= \{e_v\mid v\in X^{(0)}\}
\end{align*}
where $v\mapsto e_v$ is a bijection from $X^{(0)}$ to $E(\Gamma(X))$.  The edge $e_v$ connects the respective components of $v$ in $\pinv(K(A))\cap X$ and $\pinv(K(B))\cap X$.

Then the equality \[\krank(H) = \rank(\pi_1(\Gamma(X)))+\#\{Y\in V(\Gamma(X))\mid \pi_1(Y)\neq \{1\}\}\] holds.  In particular, if $X$ has finitely many vertices, then $\Gamma(X)$ is finite and hence $\krank(H)<\infty$.
\end{Lemma}
\begin{proof}
Corollary~\ref{pi1inj} shows that $\p_*\colon \pi_1(X,w_0)\to H$ is an isomorphism.  The standard covering space proof of the Kurosh theorem shows that $X$ is homotopy equivalent to a graph of spaces  with underlying graph $\Gamma(X)$ where the vertex space associated to $Y\in V(\Gamma(X))$ is the space $Y$ and the edge space associated to $e_v$ is the one-point space $\{v\}$ with the obvious inclusions of edge spaces into vertex spaces.  This gives a free product decomposition of $\pi_1(X,w_0)$ which maps under $\p_*$ to a Kurosh decomposition of $H$.  One has $\pi_1(\Gamma(X))\cong F(H)$, the $H_i$ are the $p_*(\pi_1(Y))$ with $Y\in \pi_0(\pinv(K(A))\cap X)$ and the $K_j$ are the $p_*(\pi_1(Z))$ with $Z\in \pi_0(\pinv(K(B))\cap X)$ (with appropriate choices of base points).
\end{proof}

It is, in fact, not difficult to prove that $\Gamma(\core(H))\cong H\backslash T_H$ where $T_H$ is the unique minimal $H$-invariant subtree of the Bass-Serre tree $T$ for the splitting $G=A\ast B$.

The Kurosh rank is well known not to change under conjugation.  We shall need the second statement of the following corollary.

\begin{Cor}\label{conjugaterank}
If $H\leq G=A\ast B$ has finite Kurosh rank, then each conjugate $gHg\inv$ of $H$ has finite Kurosh rank.  Moreover, if $v$ is a vertex of $\core(H)$ and $g\in G$ is represented by the label of a reduced path in $\core(H)$ from $v$ to $w_0$, then $\core(gHg\inv)\subseteq \core(H)$.
\end{Cor}
\begin{proof}
Let $v$ be the endpoint of the lift $p$ of the reduced word representing $g\inv$ beginning at $w_0$.  Then $\p_*\colon (\til K_H,v)\to \pi_1(K,v_0)$ is the pointed covering with $\p_*(\pi_1(\til K_H,v))=gHg\inv$.  If $h\in H$ and $q$ is the reduced loop at $w_0$ with $\p_*([q])=h$, then $p\inv qp$ is a loop at $v$ with $\p_*([p\inv qp])=ghg\inv$.  It follows that if $X$ is the subcomplex of $\til K_H$ induced by the vertices of the path $p$ together with the vertices of $\core(H)$, then $\p_*(\pi_1(X,v))=gHg\inv$.  Thus $\core(gHg\inv)\subseteq X$ by Corollary~\ref{smallestinduced}. But $X$ has finitely many vertices by construction.  Thus $gHg\inv$ has finite Kurosh rank.  Moreover, if $v$ is a vertex of $\core(H)$ and $p$ is chosen to be a reduced path in $\core(H)$, then $X=\core(H)$ in this case.  This proves the second statement.
\end{proof}

The next proposition generalizes that a Stallings graph of a finitely generated subgroup of a free group is a Schreier graph if and only if the subgroup has finite index.

\begin{Prop}\label{completecomplex}
Let $H\leq G=A\ast B$ be a subgroup of finite Kurosh rank.  Then the index of $H$ in $G$ is finite if and only if $\core(H)=\til K_H$.
\end{Prop}
\begin{proof}
If $\core(H)=\til K_H$, then since $[G:H]$ is the cardinality of the vertex set of $\til K_H$, we conclude that $H$ has finite index. Suppose conversely that $[G:H]=n<\infty$ and let $v$ be a vertex of $\til K_H$.  Chooses a reduced path $p$ from $w_0$ to $v$ and let $\gamma\in (\til A\cup \til B)^*$ be the label of $p$.  If $\gamma$ is not cyclically reduced, then we can find $c\in \til A\cup \til B$ such that $\gamma c$ is cyclically reduced.  Hence there is a cyclically reduced word $\beta$ such that the lift of $\beta$ at $w_0$ passes through $v$.  Now $\beta^{n!}$ is a reduced word representing an element of $H$.   Thus $v$ is a vertex of $\core(H)$.  We conclude $\core(H)=\til K_H$.
\end{proof}

\section{The proof of Theorem~\ref{main}}

In order to prove Theorem~\ref{main} we shall need some technical notions and results.  Let us assume that $H\leq  G=A\ast B$ is a subgroup of finite Kurosh rank.  We retain the notation of the previous section.

We now associate to each $g\in G$ a partial permutation of the finite set $V$ of vertices of $\core(H)$.
If $v$ is a vertex of $\til K_H$ and $g\in G$, then $vg$ denotes the endpoint of the lift of $\red(g)$ starting at $v$;  this is the right monodromy action of $G$ on $\til K_H^0$.  Let $\dom(g)$ to be the set of vertices $v\in V$ such that the lift of $\red(g)$ at $v$ is contained in $\core(H)$.  Define $\ran(g)=\{vg\mid v\in \dom(g)\}$ and note that $\rho_g\colon \dom(g)\to \ran(g)$ given by $v\mapsto vg$ is a bijection.  Also observe that $\rho_{g\inv}$ is the inverse partial permutation to $\rho_g$.  Define $\rank(g)=\#\dom(g)=\#\ran(g)$.  Clearly, we have $\rank(g)=\rank(g\inv)$.

If $\gamma\in (\til A\cup \til B)^*$, denote by $\ov{\gamma}$ the element of $G$ represented by $\gamma$.

\begin{Prop}\label{factorrank}
The following properties hold.
\begin{itemize}
\item [(a)] If $\alpha\beta\in (\til A\cup \til B)^*$ is reduced, then $\dom(\ov{\alpha\beta})\subseteq \dom(\ov{\alpha})$.
\item [(b)] If $\alpha\beta\in (\til A\cup \til B)^*$ is reduced, then $\ran(\ov{\alpha\beta})\subseteq \ran(\ov{\beta})$.
\item [(c)] If $\alpha\beta\gamma\in (\til A\cup \til B)^*$ is reduced,  then $\rank(\ov {\alpha\beta\gamma})\leq \rank (\ov \beta)$.
\end{itemize}
\end{Prop}
\begin{proof}
By uniqueness of path lifting, if the lift of $\alpha\beta$ at $v$ is contained in $\core(H)$, then the same is true for the lift of $\alpha$ at $v$. This proves (a).  Item (b) is dual.  The final claim follows by (a) and (b) because $\rank(\ov {\alpha\beta\gamma})\leq \rank(\ov{\alpha\beta})\leq \rank(\ov \beta)$.
\end{proof}

%The rank is also preserved under taking cyclic conjugates of cyclically reduced words.  More precisely, we have the following.
%
%\begin{Prop}\label{cyclicconj}
%Suppose that $\alpha\beta\in (\til A\cup \til B)^*$ is reduced.  Then $\rank(\ov{\alpha\beta})\leq \rank(\ov{\beta\alpha})$.  Moreover, equality holds if $\beta\alpha$ is reduced.
%\end{Prop}
%\begin{proof}
%Define an injective map $f\colon \dom(\ov{\alpha\beta})\to \dom(\ov{\beta\alpha})$ by $f(v)=v\ov\alpha$.
%\begin{figure}[bt]
%\begin{center}
%\begin{tikzpicture}[->,shorten >=1pt,%
%auto,node distance=2cm,semithick,
%inner sep=1pt,bend angle=45]
%\tikzset{every state/.style={minimum size=0pt}}.
%\node[state] (A) {\footnotesize{$\,\,v\,\,$}};
%\node[state] (B) [right of=A] {\scriptsize{$v\ov \alpha$}};
%\tikzstyle{every node}=[font=\footnotesize]
%\path (A) edge [bend left] node  {$\alpha$} (B)
%      (B) edge [bend left] node {$\beta$} (A);
%\end{tikzpicture}
%\end{center}
%\caption{Cyclic conjugates\label{cyccongfig}}
%\end{figure}
%Since $\beta\alpha$ labels a path in $\core(H)$ from $v\ov \alpha$ to $v$, it follows from Lemma~\ref{stayininduced} that $\red(\beta\alpha)$ labels a path in $\core(H)$ from $v\ov \alpha$ to $v$.  See Figure~\ref{cyccongfig}.  Thus $v\ov\alpha\in \dom(\ov{\beta\alpha})$.  Clearly $f$ is injective, being the restriction of $\rho_{\ov \alpha}$ to $\dom(\ov{\alpha\beta})\subseteq \dom(\ov \alpha)$. Thus $\rank(\ov{\alpha\beta})\leq \rank(\ov{\beta\alpha})$.
%
%If $\beta\alpha$ is reduced, then by symmetry we obtain the reverse inequality and hence the conclusion.
%\end{proof}

Our next lemma follows from the construction of $\core(H)$.

\begin{Lemma}\label{degree2}
Let $v\in V$ with $v\neq w_0$.  Then there are reduced paths $p_A(v)$ and $p_B(v)$ from $v$ to $w_0$ such that the first edge of $p_A(v)$ is labelled by a letter from $\til A$ and the first edge of $p_B$ is labelled by a letter from $\til B$.
\end{Lemma}
\begin{proof}
By construction of $\core(H)$, there is an element $g\in H$ such that the lift $q$ of $\red(g)$ at $w_0$ passes through $v$.  Say that $q=q_1q_2$ with the endpoint of $q_1$ being $v$.  Replacing $g$ by $g\inv$, if necessary, we may assume that the last edge of $q_1$ is labelled by an element of $\til A$ and the first edge of $q_2$ is labelled by an element of $\til B$.  Then we may take $p_A(v)=q_1\inv$ and $p_B(v)=q_2$.
\end{proof}

If $H\leq G$ is of finite Kurosh rank and infinite index, then $\core(H)\neq \til K_H$ by Proposition~\ref{completecomplex}.  Since $\til K_H$ is connected, there must be an edge $e$ of $\til K_H$ such that one endpoint of $e$ belongs to $\core(H)$ and the other endpoint does not.  Let us say that $H$ is \emph{well situated} if there is an edge $e$ at $w_0$ whose other endpoint does not belong to $\core(H)$. We introduce this notion only to simplify the argument, not for any essential purpose.

\begin{Lemma}\label{wellsituated}
Let $H\leq G$ be an infinite index subgroup of finite Kurosh rank.  Then there is a conjugate of $H$ which is well situated.
\end{Lemma}
\begin{proof}
Let $e$ be an edge of $\til K_H$ with one endpoint $v$ in $\core(H)$ and one endpoint outside $\core(H)$.  Then $K=\p_*(\pi_1(\til K_H,v))$ is a conjugate of $H$ and by Corollary~\ref{conjugaterank} we have $\core(K)$ is a subcomplex of $\core(H)$ with base point $v$.  Thus $K$ is well situated.
\end{proof}

The following technical lemma is the key to our proof of the conjecture of Karass and Solitar.

\begin{Lemma}\label{rank0}
Let $H\leq G$ be an infinite index subgroup of finite Kurosh rank which is well situated.  Then there is a cyclically reduced element $g\in G$ with $\rank(g)=0$.
\end{Lemma}
\begin{proof}
Since $H$ is well situated, there is an edge $e$ at $w_0$ whose endpoint is not in $\core(H)$.  Without loss of generality, we may assume that $e$ is labelled by $a\in \til A$.  We now show that there is an element $g\in G$ of rank zero.

Choose $g\in G$ of minimum rank.  Multiplying $g$ on the right by an element of $\til B$, if necessary, and using Proposition~\ref{factorrank}, we may assume without loss of generality that the last symbol in $g$ is from $\til B$.  Suppose that $\rank(g)>0$ and let $v\in \dom(g)$. Let $p$ be the lift of $\red(g)$ at $v$.

There are two cases.  Suppose first that $vg=w_0$.  Then $p$ ends at $w_0$ and $pe$ is a reduced path with label $ga$.  Thus $v\notin \dom(ga)$.  On the other hand, Proposition~\ref{factorrank} implies that $\dom(ga)\subseteq \dom(g)$.  We conclude $\rank(ga)<\rank(g)$, a contradiction.

Next suppose that $vg\neq w_0$.  Lemma~\ref{degree2} implies that there is a reduced path $q$ from $vg$ to $w_0$ such that the first symbol of the label of $q$ belongs to $\til A$. Let $g'$ be the element of $G$ represented by the label of $q$.  Then $\red(gg')=\red(g)\red(g')$ and $\ell(g')\geq 1$.  Thus $\red(g)$ is still a prefix of $\red(gg'a)$ and hence $\dom(gg'a)\subseteq \dom(g)$ by Proposition~\ref{factorrank}.   But $\red(gg'a)$ labels the path $\red(pqe)$ from $v$ to $w_0a\notin \core(H)$. Therefore, $v\notin \dom(gg'a)$ and  so $\rank(gg'a)<\rank(g)$, again a contradiction.  We conclude that $\rank(g)=0$.

If $g$ is cyclically reduced, we are finished.  Otherwise, by right multiplying $g$ by an appropriate element of $\til A\cup \til B$ we obtain a cyclically reduced element, which has rank zero by an application of Proposition~\ref{factorrank}.
\end{proof}

We are now ready to prove Theorem~\ref{main} and its corollary, Theorem~\ref{main2}.
\begin{proof}[Proof of Theorem~\ref{main}]
Suppose $H$ has finite index.  Then it contains a non-trivial normal
subgroup, and in a free product any two non-trivial normal subgroups
have non-trivial intersection, cf.~\cite[pg.~211]{KS1}.

Next suppose that $H$ has infinite index and finite Kurosh rank.  We must show that there is a non-trivial normal subgroup $N$ with $H\cap N=1$.  Replacing $H$ by a conjugate, we may assume that $H$ is well situated by Lemma~\ref{wellsituated}.  By Lemma~\ref{rank0}, there is a cyclically reduced element $g$ of rank zero.  Let $N$ be the normal subgroup generated by $g^6$. We claim that $H\cap N=\{1\}$.  First note that if $1\neq z\in H\cap N$, then $\red(z)$ lifts to a loop at $w_0$ in $\core(H)$.  Thus it suffices to show that there is no $z\in N\setminus \{1\}$ such that $\red(z)$ lifts to a loop in $\core(H)$ at some vertex.

Suppose that $z\in N\setminus \{1\}$ is such that $\red(z)$ lifts to a loop in $\core(H)$ at some vertex $v$. Then $\red(z)=\red(u)\red(z')\red(u)\inv$ with $\red(z')$ cyclically reduced.  Clearly, $\red(z')$ lifts to a loop in $\core(H)$ at the vertex $vu$ and so without loss of generality we may assume that $z$ is cyclically reduced.
By~\cite[Theorem~3.1]{DuncanHowie}, there is a cyclic conjugate $\zeta$ of $\red(z)$ containing a factor $\gamma$ of length at least $3\ell(g)-1>\ell(g)$ that is also a factor of a cyclic conjugate of $\red(g)^{\pm 6}$.  Thus $\gamma$ has $\red(g)^{\pm 1}$ as a factor. Now $\rank(g\inv)=\rank(g)=0$, and so $\rank(\ov \gamma)=0$ by Proposition~\ref{factorrank}.  Therefore, $\rank(\ov \zeta)=0$ (again by Proposition~\ref{factorrank}).  But since $z$ labels a loop in $\core(H)$ at some vertex, the same is clearly true of its cyclic conjugate $\zeta$.  Thus $\rank(\ov \zeta)>0$.  This contradiction completes the proof.
\end{proof}

\def\malce{\mathbin{\hbox{$\bigcirc$\rlap{\kern-7.75pt\raise0,50pt\hbox{${\tt
  m}$}}}}}\def\cprime{$'$} \def\cprime{$'$} \def\cprime{$'$} \def\cprime{$'$}
  \def\cprime{$'$} \def\cprime{$'$} \def\cprime{$'$} \def\cprime{$'$}
  \def\cprime{$'$}

%\bibliographystyle{abbrv}
%\bibliography{standard2}

\end{document}